\documentclass[11pt]{article}
\usepackage{setspace,fancyhdr}
\usepackage{mathrsfs}
\usepackage{amsmath}
\usepackage{amsthm}
\usepackage{epsfig}
\usepackage{graphicx}
\usepackage{amsmath,mathtools}
\usepackage{color}
\usepackage[margin=1in]{geometry}
\usepackage{xspace}

\usepackage{amssymb,amsmath,amsfonts}

\newtheorem{theorem}{Theorem}[section]

\numberwithin{equation}{section}

\newcommand{\eq}{\begin{equation}\begin{array}{lclllllllllllllll}}
\newcommand{\ee}{\end{array}\end{equation}}

\renewcommand{\thefootnote}{\arabic{footnote}}
\newcommand{\bmt}{\left[ \begin{array}{ccccccccc}}
\newcommand{\emt}{\end{array}\right]}

\newcommand{\bea}{\begin{eqnarray}}
\newcommand{\eea}{\end{eqnarray}}
\newcommand{\bean}{\begin{eqnarray*}}
\newcommand{\eean}{\end{eqnarray*}}
\newcommand{\MA}{Monge-Amp\`ere\xspace}

\newcommand{\IR}{{\mathbb{R}}}

\begin{document}

\title{A multigrid scheme for 3D Monge-Amp\`ere equations
\thanks{Project was supported in part by Guangdong Province Science and Technology Development Grant (foreign cooperation project: 2013B051000075) of China, and in part by NSF 1021203, 1419028  of the United States.}
}

 \author{
 Jun Liu \footnotemark[2]     \footnotemark[6]    
  \and
  Brittany D. Froese \footnotemark[3]
 \and
Adam M. Oberman  \footnotemark[4]  \footnotemark[6]    
\and
Mingqing Xiao \footnotemark[5]
 }

\date{\today }

\maketitle

\renewcommand{\thefootnote}{\fnsymbol{footnote}}
\footnotetext[6]{The work described in this article is a result of a collaboration made possible by the IMA's New Directions Short Course on ``Topics on Control Theory: Optimal Mass Transportation'' in June, 2014.
}
\footnotetext[2]{Corresponding author. Department of Mathematics and Statistical Sciences, Jackson State University, Jackson, MS 39217, USA.
({\tt jun.liu@jsums.edu})}
\footnotetext[3]{Department of Mathematical Sciences, New Jersey Institute of Technology, University Heights, Newark, NJ  07102, USA. 
(\tt bdfroese@njit.edu)}
\footnotetext[4]{Department of Mathematics and Statistics, McGill University, Montreal, QC H3A 0B9, Canada. 
({\tt adam.oberman@mcgill.ca})}
\footnotetext[5]{Department of Mathematics, Southern Illinois University, Carbondale, IL 62901, USA. 
({\tt mxiao@siu.edu})}

\renewcommand{\thefootnote}{\arabic{footnote}}
\begin{abstract}
The elliptic Monge-Amp\`ere equation 
is a fully nonlinear partial differential equation which  has been the focus of increasing attention from the scientific computing community.  Fast three dimensional solvers are needed, for example in medical image registration
but are not yet available. We build fast solvers for smooth solutions in three dimensions using a nonlinear full-approximation storage multigrid method. Starting from a second-order accurate centered finite difference approximation, 
we present a nonlinear Gauss-Seidel iterative method which has a mechanism for selecting the convex solution of the equation.  The iterative method is used as an effective smoother,  combined with the full-approximation storage multigrid method. 
Numerical experiments are provided to validate the accuracy of the finite difference scheme and illustrate the computational efficiency of the proposed multigrid solver.   
\end{abstract}

\providecommand{\keywords}[1]{\textbf{Keywords:} #1}

\section{Introduction}\label{sec1}
The elliptic Monge-Amp\`ere equation 
\[
\begin{cases}
 \det(D^2 u(x))=f(x), & x\in\Omega\subset\IR^d\\
 u \text{ is convex}
\end{cases}
\]
with $f>0$, where $D^2 u$ is the Hessian of the function $u$,
is a fully nonlinear partial differential equation (PDE)
which  has been the focus of increasing attention from the scientific computing community~\cite{Tadmor2012,Feng2013}.   It is the prototypical fully nonlinear elliptic PDE, with connections to differential geometry~\cite{trudinger2008monge}.  As tools for solving the equation have become more effective, the PDE has appeared in more applications, including the reflector design problem~\cite{Prins2014}, the optimal transportation problem~\cite{Benamou2014}, and parameter identification problems in seismic signals~\cite{engquist2014application}.  
In this paper, we focus on building fast solvers for the PDE in the three dimensional case.  
Our interest in this case stems   from the image registration problem~\cite{haker2004optimal,haber2010efficient,haber20093d,haber2007image},
where current imaging resolution requires more effective solvers than are currently available.

There are two regimes to consider when solving the problem. In the regime where solutions are smooth, many approaches can lead to numerical  methods which converge in practice.
These include finite differences~\cite{DelzannoChacon,Loeper2005,Sulman2011},  finite elements~\cite{Brenner2011,Brenner2012,feng2009vanishing,Neilan2013}, spectral methods~\cite{saumier2013efficient}, and Fourier integral formulations~\cite{zheligovsky2010monge}.

In the second regime, where solutions are singular, standard methods either break down completely or become extremely slow~\cite{Benamou2010}.   
In order to compute singular solutions, the theory of viscosity solutions can be used to build  convergent monotone finite difference methods~\cite{BSNum}.   Monotone schemes require using wide stencils, which have a directional resolution parameter, in addition to a spatial discretization parameter.  As a result, they are generally less accurate {than} second order in space.  On the other hand, they are stable in the maximum norm, and can be solved by a (slow) forward Euler method, with an explicit CFL condition~\cite{Oberman2008}.     Refinements of the monotone method result in faster solvers~\cite{ObermanFroeseFast} and higher accuracy~\cite{ObermanFroeseFiltered}.  
The fast solvers employed were an exact Newton's method (with an explicitly computed Jacobian) and a direct linear solver in two dimensions~\cite{ObermanFroeseFiltered}.   The initial guess for Newton's method used one step of an iterative solver from~\cite{Benamou2010}, which also required a single solution of a Poisson equation.  Computations were performed in three dimensions, but the direct linear solvers were too costly for larger sized problems.    

A significant challenge in computing solutions of the \MA equation is the need to enforce an additional constraint that the solution be convex; without this, the solution is not unique.  In two-dimensions, for example, the \MA equation will typically have two solutions: one convex and one concave.  The situation becomes more complicated at the discrete level.  For a two-dimensional problem with $N$ discretization points, a finite difference scheme will typically admit $2^N$ different solutions~\cite{Feng2013}.  In three-dimensions, the problem becomes still more complicated.  In two-{dimensions}, the correct solution can be selected by choosing the solution that has a positive Laplacian.  This is equivalent to selecting the smallest root of a quadratic equation.  {This structure was exploited by the Gauss-Seidel and iterative Poisson methods developed in~\cite{Benamou2010}.  However, in three-dimensions, positivity of the Laplacian (sum of the eigenvalues of the Hessian) and the \MA operator (product of the eigenvalues) is not enough to guarantee convexity (positivity of all three eigenvalues).}  Using  monotone schemes, convexity can be enforced by separately computing the positive and negative parts of various second directional derivatives.


In this paper, we focus on fast solution methods for the three-dimensional \MA operator,
with a given source function which is positive and continuous in a convex domain $\Omega\subset \IR^3$,
\begin{equation}\label{MA3D}
\det (D^2u(x,y,z))=f(x,y,z) > 0, \quad \text{for $(x,y,z)$ in } \Omega
\end{equation}
with a given Dirichlet boundary condition
\begin{equation}\label{DBC}
 u =g ,\quad \text{on} \quad \partial\Omega
\end{equation}
and the additional global constraint 
\[u \text{ is convex},\]
 which is necessary for ellipticity and uniqueness.
We record the the explicit form of the operator in three dimensions as follows
 \[
\det (D^2u ) = u_{xx}u_{yy}u_{zz}+2u_{xy}u_{yz}u_{xz}-u_{xx}(u_{yz})^2-u_{yy}(u_{xz})^2-u_{zz}(u_{xy})^2.
\]
We note that the convexity constraint is formally equivalent to the condition that the Hessian of the solution is positive definite,
\[ D^2u > 0. \]

We assume also that that boundary data $g$ is consistent with the restriction of a convex function to the boundary.
These conditions on the data, along with strict convexity of the domain, are enough to ensure that solutions are classical (twice continuously differentiable)~\cite{CaffarelliNirenbergSpruckMA}.  For simplicity of computations, we work with square domains, which allow for the possibility of singular solutions, however, using known solutions avoids this problem.

Without additional  difficulties, the method developed in this paper can also be applied to the two-dimensional problem
\begin{equation}\label{MA2D}
u_{xx}u_{yy}-(u_{xy})^2=f(x,y)\quad \text{in} \quad \Omega.
\end{equation}

The objective of this work is to build fast solvers for the three-dimensional elliptic Monge-Ampere equation.  We focus on the Dirichlet problem in the smooth solution regime, where the simple narrow stencil finite difference methods can be used.   
The goal is to build an effective, fast, scalable solver that can be used in applications.
It is not trivial to build a  multigrid solver in this setting--naive implementations of multigrid will fail.
For example, the commonly suggested strategy~\cite{Trottenberg2001} of performing a few
Newton iterations to approximately solve the corresponding nonlinear scalar equations at each grid point would not lead to a convergent FAS
multigrid solver for our problem. This is because such a method has no mechanism for selecting the correct, convex solutions.
We introduce such a selection {mechanism}, which leads to an effective multigrid method.  
However, when the equation becomes degenerate (corresponding to non-strictly convex solutions, which can arise when the right hand side $f(x) = 0$), the method presented here breaks down.  This is likely related to the loss of uniform ellipticity of the equation, and the loss of strict convexity of the solution~\cite{Benamou2010}.  {In this setting, wide-stencil finite difference methods are typically needed to guarantee convergence to weak solutions~\cite{Froese2011,Motzkin}.}  A possible work around is to combine monotone (or filtered) schemes with a multi grid method.   For now, we restrict our attention to smooth solutions, where the method presented in this paper is effective.

One approach is to use Newton's method  with an iterative solver~\cite{Ortega2000}.
For example,  the Newton-Multigrid method arises when  a 
linear multigrid algorithm is used to approximately solve the linearized Jacobian system.  A disadvantage of Newton's method is that it requires a good initial guess, which should typically be convex~\cite{Loeper2005}.
We instead choose to treat the non-linearity directly using the linear multigrid framework and coarse-grid correction,
which leads to the most common nonlinear multigrid method---the full approximation scheme (FAS)~\cite{Brandt1977,Brandt2011multigrid,Briggs2000,Trottenberg2001,Saad2003}. 
For the FAS multigrid to be effective as an iterative method, it requires an effective
smoother (or relaxation) scheme such as a nonlinear Gauss-Seidel iteration, which must eliminate
the high frequency components of the approximation errors at the current fine level.
Once the FAS V-cycle iteration is established, the full multigrid (FMG) technique, 
based on nested iterations, can be exploited to obtain a good initial guess for the fine-grid problem.

In this work we develop a FMG-FAS multigrid algorithm for the 3D Monge-Amp\`ere equation.  The method is based on a nonlinear Gauss-Seidel iteration that includes a mechanism for selecting the convex solution.
As a noticeable advantage over the above mentioned Newton-Multigrid method,
the FMG-FAS multigrid algorithm usually has better global convergence properties.  This is particularly the case for the \MA equation since the Newton-Multigrid approach does not have any way to select the correct, convex solution.
The FMG-FAS multigrid typically has a lower memory requirement as well since there is no need to compute and store the Jacobian matrices~\cite{Trottenberg2001}.

There are very few published works devoted to designing fast nonlinear solvers, particularly for larger problem sizes in three dimensions. 
The multigrid method was used as a preconditioner for a Newton-Krylov iteration in~\cite{DelzannoChacon,delzanno2008optimal} in two and three dimensions up to grids of size $128^3$, which took approximately five minutes using eight processors.  For the applications considered in that work, solutions of the \MA equation would typically be close to $u = |x|^2/2$, which led naturally to a good initialization of Newton's method.
In this paper, we approach the problem using the nonlinear full approximation scheme (FAS) multigrid method, which will not require an accurate initialization.
Although the FAS multigrid methods are popular and widely used in many disciplines involving nonlinear PDEs,
such as variational image registration models~\cite{Chumchob2011},  
the only references we found on applying FAS multigrid method to the Monge-Amp\`ere equation were two-dimensional simulations in a master's 
thesis~\cite{Ashmeen2011}. Besides, an adaptive FAS multigrid algorithm based on a continuation method was developed in
~\cite{Chen2010} for solving a 2D balanced vortex model.  
It is worthwhile to point out that the reported computational CPU time does not scale linearly with respect to the degrees of 
freedom (see e.g. Table 5.1 in~\cite{Ashmeen2011}),
which motivates us to look further into the efficiency of the FAS multigrid implementations.
In addition, we are mainly interested in solving three dimensional (3D) Monge-Amp\`ere equation,
which is well-known to be much more difficult~\cite{Brenner2012}.


This paper is organized as follows.
In the next section, we present a discretization of the PDE using second-order centered finite differences and propose a nonlinear Gauss-Seidel iterative method for solving the discretized system.   The method extends the two-dimensional Gauss-Seidel method of~\cite{Benamou2010}; in our case, we need to solve a cubic rather than a quadratic equation.
In Section~\ref{secFAS}, we introduce a full multigrid method based on the FAS V-cycle algorithm, where the nonlinear Gauss-Seidel iterations function as a smoother.
In Section~\ref{secNum}, numerical results for several two- and three-dimensional examples
are reported, which demonstrate the accuracy, mesh independent convergence, and linear time complexity of our proposed FMG-FAS multigrid solver.
Finally, {the paper ends with several concluding remarks} in Section~\ref{secFinal}.

\section{A nonlinear Gauss-Seidel iteration in 3D}
\label{secGS}
This section is motivated by the two-dimensional explicit finite difference method in~\cite{Benamou2010}
where second derivatives are discretized using standard centered differences on a uniform 
Cartesian grid. {There} a Gauss-Seidel iteration is constructed by selecting the smaller root
of a point-wise defined quadratic equation over each grid point. 
The resulting method is {formally} second order accurate provided the solution is regular enough.
In the following, we will develop a similar Gauss-Seidel iteration for the three-dimensional case,
where we have to handle a cubic equation at each grid point.

Let $N>0$ be a positive integer. We discretize the space domain $\overline\Omega=[0,1]^3$ uniformly into 
\[
 \Omega_h=\left\{(x_i,y_j,z_k) \mid  x_i=ih, y_j=jh, z_k=kh;\  i,j,k = 1,\dots, N \right\}
\]
with a uniform mesh step size $h=1/N$.
Let $u_{i,j,k}$ be the discrete approximation of $u(x_i,y_j,z_k)$.
Notice that the values of $u_{i,j,k}$ with $i\in\{0,N\}$ or $j\in\{0,N\}$ or $k\in\{0,N\}$ are
directly specified by the Dirichlet boundary conditions. 
Hence, we only need to set up finite difference approximations for all the interior grid points.
We employ the second-order accurate centered finite difference approximations~\cite{LeVeque2007} for all the
second-order derivatives as follows.
\begin{align}
 u_{xx}^h(x_i,y_j,z_k)&= \frac{1}{h^2}\left( {u_{i+1,j,k}+u_{i-1,j,k}} -2u_{i,j,k} \right),\\
 u_{yy}^h(x_i,y_j,z_k)&= \frac{1}{h^2}\left( {u_{i,j+1,k}+u_{i,j-1,k}} -2u_{i,j,k} \right),\\
 u_{zz}^h(x_i,y_j,z_k)&= \frac{1}{h^2}\left( {u_{i,j,k+1}+u_{i,j,k-1}} -2u_{i,j,k} \right),\\
 u_{xy}^h(x_i,y_j,z_k)&= \frac{1}{4h^2}\left(u_{i+1,j+1,k}+u_{i-1,j-1,k}-u_{i-1,j+1,k}-u_{i+1,j-1,k} \right),\\
 u_{xz}^h(x_i,y_j,z_k)&= \frac{1}{4h^2}\left(u_{i+1,j,k+1}+u_{i-1,j,k-1}-u_{i-1,j,k+1}-u_{i+1,j,k-1} \right),\\
 u_{yz}^h(x_i,y_j,z_k)&= \frac{1}{4h^2}\left(u_{i,j+1,k+1}+u_{i,j-1,k-1}-u_{i,j+1,k-1}-u_{i,j-1,k+1} \right).
\end{align}
Using these approximations, we can construct the discrete Hessian
\[ D^2u^h = \left(\begin{tabular}{ccc}$u_{xx}^h$ & $u_{xy}^h$  & $u_{xz}^h$\\ $u_{xy}^h$ & $u_{yy}^h$ & $u_{yz}^h$\\ $u_{xz}^h$ & $u_{yz}^h$ & $u_{zz}^h$\end{tabular}\right). \]
Then the discrete form of the elliptic \MA equation in the interior of the domain reads
\begin{equation}\label{eq:MA_discrete}
\begin{cases}
\det(D^2u^h_{i,j,k}) = f_{i,j,k}\\
D^2u^h_{i,j,k} > 0.
\end{cases}
\end{equation}
Inserting the approximations into equation~\eqref{eq:MA_discrete} at each interior grid point $(x_i,y_j,z_k)$ leads to
a cubic polynomial $P_3(u_{i,j,k};u)$ with respect to $u_{i,j,k}$
\begin{align}\label{MA3Dcubic}
 P_3(u_{i,j,k};u):=&-8u_{i,j,k}^3+4(a+b+c) u_{i,j,k}^2+2(r^2+s^2+t^2-ab-ac-bc)u_{i,j,k} \nonumber \\
 &\qquad +(abc-at^2-bs^2-cr^2+2rst-h^6f_{i,j,k})=0,
\end{align}
 where we use the notation
 \begin{align*}
  a&=({u_{i+1,j,k}+u_{i-1,j,k}}),\\
  b&=({u_{i,j+1,k}+u_{i,j-1,k}}),\\
  c&=({u_{i,j,k+1}+u_{i,j,k-1}}),\\
  r&=\frac{1}{4}\left(u_{i+1,j+1,k}+u_{i-1,j-1,k}-u_{i-1,j+1,k}-u_{i+1,j-1,k} \right),\\
  s&=\frac{1}{4}\left(u_{i+1,j,k+1}+u_{i-1,j,k-1}-u_{i-1,j,k+1}-u_{i+1,j,k-1} \right),\\
  t&=\frac{1}{4}\left(u_{i,j+1,k+1}+u_{i,j-1,k-1}-u_{i,j+1,k-1}-u_{i,j-1,k+1} \right),
 \end{align*}
and $f_{i,j,k}=f(x_i,y_j,z_k)$.

The above finite difference approximations (\ref{MA3Dcubic}) are intentionally formulated as a cubic equation
of $u_{i,j,k}$ at the current node $(x_i,y_j,z_k)$ by assuming its neighboring nodes are fixed.
The nonlinear Gauss-Seidel iteration follows from iteratively solving the cubic equation sequentially
over all grid points in a lexicographic order.
The so-called red-black ordering with better smoothing and parallel properties
is usually recommended for 2D problems, but it is of limited advantage  \cite{Zhang1998} to our 3D scheme
involving a total of 19 points. 
In this case, four or more colors are necessary to fully parallelize the updating order~\cite{Gupta2000},
which is not further pursued in the current paper.
The Dirichlet boundary conditions (\ref{DBC})
are enforced at the boundary grid points during the iteration.
However, we also need to determine which root of the cubic is to be selected
in order to ensure local convexity of the discrete solutions.
In the 2D case, the corresponding approach leads to a quadratic equation~\cite{Benamou2010}, 
and selection of the smallest root yields the convex solution.
This rule is not directly applicable for the cubic equation arising in 3D, which may have either one or three real roots.  As we are only interested in real-valued solutions, we simply select the smallest of all the real roots; that is
\begin{align}\label{realroots}
 u^{n+1}_{i,j,k}=\min\Big\{\mathrm{real\ roots\ of}\ P_3(\cdot\,;u^n)\Big\}.
\end{align}
These roots can either be computed exactly or approximately using a root-finding algorithm.

{Fixed points of this scheme are equivalent to solutions of the \MA equation~\eqref{eq:MA_discrete}, which is justified in Theorems~\ref{thm:root}-\ref{thm:root2}.  In particular, we emphasize that the fixed point of this scheme is guaranteed to be discretely convex.  This provides a clear advantage over Newton's method, which must be coupled to a sufficiently accurate initial guess in order to prevent convergence to an incorrect, non-convex solution.}

\begin{theorem}\label{thm:root}
Let $u$ be a solution of the discrete \MA equation~\eqref{eq:MA_discrete}.  Then $u$ is a fixed point of~\eqref{realroots}.
\end{theorem}

\begin{proof}
Consider a fixed location $(x_i,y_j,z_k)$ in the grid.  Since $u_{i,j,k}$ satisfies the discrete \MA equation, it automatically satisfies the cubic equation $P_3(u_{i,j,k};u) = 0$ as well.  It remains to show that $u_{i,j,k}$ is the smallest real root of this cubic.

One possibility is that the cubic equation has only a single real root, in which case this must coincide with the real-valued $u_{i,j,k}$.

The other option is that the cubic equation has three real roots, $v_1 \leq v_2 \leq v_3$.  We remark that using the notation defined above, the discrete Hessian corresponding to each of these roots can be written as
\begin{equation}\label{eq:discreteHessian}
D^2(u^h;v_m) = \left(\begin{tabular}{ccc}$\frac{a^h-2v_m}{h^2}$ & $u_{xy}^h$  & $u_{xz}^h$\\ $u_{xy}^h$ & $\frac{b^h-2v_m}{h^2}$ & $u_{yz}^h$\\ $u_{xz}^h$ & $u_{yz}^h$ & $\frac{c^h-2v_m}{h^2}$\end{tabular}\right), \quad m = 1, 2, 3
\end{equation}
where the off-diagonal elements do not depend on the value of $v_m$.  This is a symmetric real-valued matrix, and therefore has real eigenvalues.

Now suppose that $\lambda$ is any eigenvalue of $D^2(u^h;v_1)$ {with eigenvector~$x$}.  Then we can compute
\[ D^2(u^h;v_m)x = \left(D^2(u^h;v_1)+\frac{2(v_1-v_m)}{h^2}I\right)x = \left(\lambda + \frac{2(v_1-v_m)}{h^2}\right)x. \]
Thus
\[ \lambda + \frac{2(v_1-v_m)}{h^2} \leq \lambda \]
is an eigenvalue of $D^2(u^h;v_m)$, with equality only if $v_1 = v_m$.  In particular, the eigenvalues of the discrete Hessian are decreasing functions of the root $v$.

Since by assumption, the discrete \MA equation has a root that yields a positive-definite Hessian, at least one of the roots $v_1,v_2,v_3$ will yield three positive eigenvalues.  As the eigenvalues are decreasing in $v$, the smallest root $v_1$ must yield three positive eigenvalues.

Suppose now that another root $v_m > v_1$ also produces positive eigenvalues.  Since both $v_1, v_m$ satisfy the discrete \MA equation at $(x_i,y_j,z_k)$, we have
\begin{align*} f_{i,j,k} &= \lambda_1(D^2(u^h;v_m)) \lambda_2(D^2(u^h;v_m)) \lambda_3(D^2(u^h;v_m))\\
 &< \lambda_1(D^2(u^h;v_1)) \lambda_2(D^2(u^h;v_1)) \lambda_3(D^2(u^h;v_1)) \\ &= f_{i,j,k}. \end{align*}  
This is a contradiction, which means that the smallest real root $v_1$ is the only root that yields three positive eigenvalues.  

We conclude that the smallest real root of the cubic is the root that corresponds to the convex solution of the discrete \MA equation~\eqref{eq:MA_discrete}.
\end{proof}

{
\begin{theorem}\label{thm:root2}
Let $u$ be a fixed point of~\eqref{realroots}.  Then $u$ is a solution of the discrete \MA equation~\eqref{eq:MA_discrete}.
\end{theorem}

\begin{proof}
Consider any fixed grid point $(x_i,y_j,z_k)$.  By the definition of the polynomial $P_3(\cdot;u^h)$, the fixed point $u_{i,j,k}$ satisfies $\det(D^2u^h_{i,j,k}) = f_{i,j,k}$.  It remains to show that the discrete Hessian $D^2u^h_{i,j,k}$ is positive definite.

Consider the symmetric real-valued matrix
\begin{equation}\label{eq:matrix}
M = \left(\begin{tabular}{ccc}$\frac{a^h}{h^2}$ & $u_{xy}^h$  & $u_{xz}^h$\\ $u_{xy}^h$ & $\frac{b^h}{h^2}$ & $u_{yz}^h$\\ $u_{xz}^h$ & $u_{yz}^h$ & $\frac{c^h}{h^2}$\end{tabular}\right), 
\end{equation}
which has three real eigenvalues $\lambda_1 \leq \lambda_2 \leq \lambda_3$.

We also define 
\begin{equation}\label{eq:discHessianFun}
D^2(u^h;v) = M - \frac{2v}{h^2} I ,
\end{equation}
which has eigenvalues $\lambda_m - \frac{2v}{h^2}$, $m = 1, 2, 3$.  We note that zeros of the polynomial $P_3(v;u^h)$ are equivalent to roots of
\[ Q(v) \equiv \det(D^2(u^h;v)) - f_{i,j,k} = \left(\lambda_1 - \frac{2v}{h^2}\right)\left(\lambda_2 - \frac{2v}{h^2}\right)\left(\lambda_3 - \frac{2v}{h^2}\right) - f_{i,j,k} = 0. \]

We record the fact that
\begin{align*}
Q\left(\frac{h^2}{2}\lambda_1\right) = - f_{i,j,k} < 0, \quad
Q\left(\frac{h^2}{2}\left(\lambda_1 - f_{i,j,k}^{1/3}\right)\right) \geq 0.
\end{align*}
By the Intermediate Value Theorem, there exists some real root~$v^* \in \left[\frac{h^2}{2}\left(\lambda_1 - f_{i,j,k}^{1/3}\right),\frac{h^2}{2}\lambda_1\right)$.  By definition, $u_{i,j,k}$ is the smallest real root and thus
\[ u_{i,j,k} \leq v^* < \frac{h^2}{2}\lambda_1. \]

Then the eigenvalues of the discrete Hessian $D^2(u^h;u_{i,j,k})$ are given by
\[ \lambda_m - \frac{2u_{i,j,k}}{h^2} > \lambda_m - \lambda_1 \geq 0, \quad m = 1,2,3 \]
and the discrete Hessian is positive definite.

We conclude that all fixed points correspond to solutions of the discrete \MA equation.
\end{proof}
}

\section{FAS multigrid method}
\label{secFAS}

The drawback of directly using the nonlinear Gauss-Seidel iteration is that the number of iterations required for convergence increases with the number of discretization points, so the total solution time grows super-linearly with the number of variables.
However,  the nonlinear Gauss-Seidel iteration can be used as an effective smoother, which makes it particularly effective when combined with the nonlinear multigrid method that follows.

In this section, we introduce the standard nonlinear FAS multigrid method 
for solving our discretized nonlinear problem. As a highly efficient iterative algorithm,
the FAS multigrid method is a nonlinear generalization of the linear multigrid algorithm, 
which typically has optimal computational complexity for linear elliptic PDEs.
The FAS multigrid method provides a powerful approach for handling nonlinear equations
without the need for the global linearization required by Newton's method.
Unlike with Newton's method, there is typically no need to initialize the solver with a very good initial guess. Below, we briefly describe the FAS multigrid algorithm as well as its FMG version
based on the standard textbooks~\cite{Brandt2011multigrid,Briggs2000,Trottenberg2001,Saad2003}.

For a general nonlinear system that is discretized using the fine mesh-size $h$
\[
 S_h (w_h)=b_h,
\]
one V-cycle FAS multigrid iteration is recursively defined in {Algorithm 1}~\cite{Briggs2000,Trottenberg2001,Saad2003}.
\begin{center}
\begin{tabular}{r|ll}
\textbf{Algorithm 1}&FAS multigrid V-cycle iteration&(with $H=2h$) \\ \hline
Steps& $w_h :=\textsf{FAS}(h,S_h,w_h^0,b_h)$\\ \hline
--& IF ($h==h_0$)& \\
 (1)& \hspace{1cm} Approximately solve: & $ S_{h_0} (w _{h_0})=b_{h_0}$\\
--& ELSE & \\
 (2)&\hspace{1cm} Pre-smooth $\nu_1$ times: & $w_h:=\textsc{smooth}^{\nu_1} (S_h, w_h^0,b_h)$ \\
 (3)&\hspace{1cm} Restriction residual:& $r_H:=I^H_h (b_h-S_h(w_h))$\\
 (4)&\hspace{1cm} Initialize coarse guess:&$u _H:= I^H_h w_h$, $w_H:=\tilde I^H_h w_h$\\
 (5)&\hspace{1cm} Define coarse r.h.s.:&  $b_H:=S_H(w_H)+r_H$\\
 (6)& \hspace{1cm} Recursion:& $u _H:=\textsf{FAS}(H,S_H,u_H,b_H)$\\
 (7)& \hspace{1cm} Prolongation: & $\delta_h:=I_H^h (u_H-w_H)$ \\
 (8)& \hspace{1cm} Correction: & $w_h:=w_h+\delta_h$\\
 (9)& \hspace{1cm} Post-smooth $\nu_2$ times:& $w_h:=\textsc{smooth}^{\nu_2} (S_h, w_h,b_h)$\\
 --& ENDIF &\\
 (10)& RETURN $w_h$. &\\ \hline
\end{tabular}
\end{center}

In a single FAS V-cycle iteration, the fine-grid solution first undergoes a small number of smoothing iterations.  Following this, the solution and residual are restricted to a coarser grid ($H = 2h$) using two (possibly different) restriction operators($I^H_h$, $\tilde I^H_h$).  A new V-cycle iteration is then performed on this coarser level, with this procedure proceeding recursively until the grid reaches the coarsest level $h_0\gg h$.
At the coarsest level, the underlying problem size has become so small that it can be (approximately) solved using a very small number of Gauss-Seidel iterations.

Computing the solution on the coarser level $H$ leads to a coarse approximation of the solution.  
A prolongation operator ($I^h_H$) transfers this approximation to 
the fine grid, which provides a coarse grid correction in the fine grid solution.  The fine grid solution is further corrected using a small number of 
smoothing iterations. 


As suggested in~\cite{Trottenberg2001},
the approximation restriction operator $\tilde I^H_h$ is chosen as straight injection.
Depending on the application, there are many possible choices for $I^H_h$ and $I_H^h$.
For a better computational efficiency,
we choose to use the residual restriction operator $I^H_h$ 
from half-weighting averaging~{\cite{Trottenberg2001}} with {the following} stencil form 
$$I^H_h=
 \frac{1}{12}\bmt
 \bmt 
 0&0&0\\
 0&1&0\\
 0&0&0\\
 \emt_h^H
 \bmt 
 0&1&0\\
 1&6&1\\
 0&1&0\\
 \emt_h^H
 \bmt 
 0&0&0\\
 0&1&0\\
 0&0&0\\
 \emt_h^H
 \emt
$$
and the prolongation operator $I^h_H$ from trilinear interpolation~{\cite{Trottenberg2001}} with {a corresponding} stencil form
$$I^h_{H}=\frac{1}{64}
 \left] \begin{array}{ccc}
\left] \begin{array}{ccc}
1 & 2 & 1 \\
2 & 4 &2 \\
1 & 2 & 1
\end{array} \right[_{H}^h
\left] \begin{array}{ccc}
2 & 4 & 2 \\
4& 8  & 4 \\
2 & 4 & 2
\end{array} \right[_{H}^h
\left] \begin{array}{ccc}
1 & 2 & 1 \\
2 & 4 &2 \\
1 & 2 & 1
\end{array} \right[_{H}^h
\end{array} \right[
$$
Other types of $I^H_h$ and $I_H^h$ are also available when better accuracy is required.

We also note that a straightforward definition of the coarse grid operator $S_H$ is possible by simply using the discretized equations with a coarser step size $H$.
This is the most common choice for a fully structured mesh.
We point out that the so-called Galerkin coarse grid operator,
which constitutes a core component of the popular linear algebraic multigrid,
is not easily used with the nonlinear FAS multigrid framework~\cite{Trottenberg2001}.

\begin{figure}[h!]\centering{
 \includegraphics[width=1\textwidth]{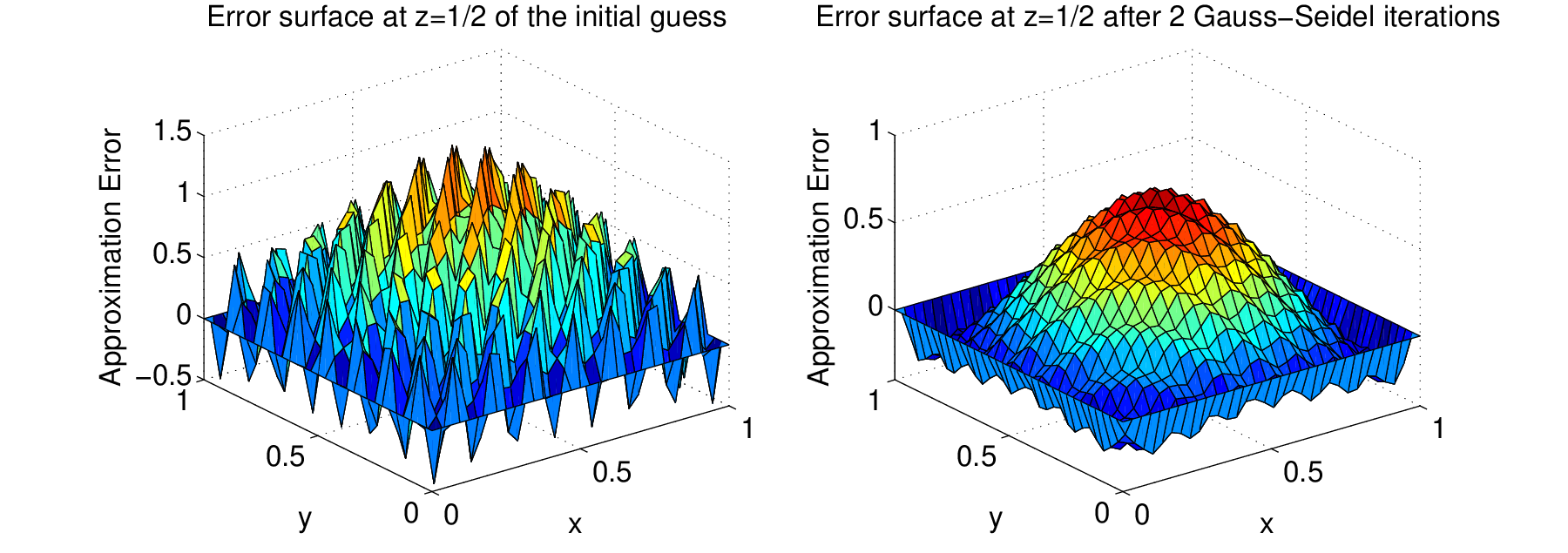}}
\caption{The smoothing effect of our nonlinear Gauss-Seidel iterations with Example 1. Left: the initial oscillating
error is chosen to be $\sin(\pi x)\sin(\pi y)\sin(\pi z)+0.5\sin(15\pi x)\sin(15\pi y)\sin(15\pi z)$
with a maximum norm $1.4712$. Right: after 2 Gauss-Seidel iterations, 
the approximation error becomes smooth with a maximum norm $0.7538$. 
Both error surface plots are only snapshots at $z=1/2$ for better visualization.
}
\label{smoothing}    
\end{figure}

The last but most crucial component of the method is an effective smoother \textsc{smooth}, 
whose major function is to eliminate
the high-frequency components of the approximation errors. 
As a standalone solver, the chosen smoothing iteration may converge very slowly as the mesh refines.
This is the case for the nonlinear Gauss-Seidel iteration developed in the previous section.  However, because it damps the high frequency components of the error, it
will serve as an effective smoother \textsc{smooth} in Algorithm~1. 
{Such a smoothing property is numerically illustrated in Figure \ref{smoothing},
which shows that two iterations seem to be sufficient to dramatically smooth out high-frequency errors.}
This is unsurprising given that the classical linear Gauss-Seidel iteration 
has been widely recognized as a benchmark smoother in the linear multigrid method.
In practice, one or two pre- and post-smoothing iterations typically give
the best overall performance. More pre- and post-smoothing iterations could be helpful
when the smoothing effect of the smoother is weak or degraded by the possible singular solutions.

The efficiency of the above FAS multigrid iterative solver can be further improved when its initial guess
is derived from the idea of nested iteration, which leads to the most efficient full multigrid (FMG)
algorithm~\cite{Briggs2000,Trottenberg2001,Saad2003}.
The FMG algorithm based on above the FAS V-cycle multigrid iteration is recursively defined in {Algorithm 2}.
\begin{center}
 \begin{tabular}{r|ll}
\textbf{Algorithm 2}& FMG-FAS algorithm & (with $H=2h$) \\ \hline
Steps& $w_h :=\textsf{FMG}(h,S_h,b_h)$\\ \hline
--& IF ($h==h_0$)& \\
 (1)& \hspace{1cm} Approximately solve: & $ S_{h_0} (w _{h_0})=b_{h_0}$\\
--& ELSE & \\
 (2)& \hspace{1cm} Restriction:& $b_H :=\tilde I_h^H b_h$\\
 (3)& \hspace{1cm} Recursion:& $w_H :=\textsf{FMG}(H,S_H,b_H)$\\
 (4)& \hspace{1cm} Cubic interpolation: & $w_h:=\hat I_H^h w_H$\\
 (5)& \hspace{1cm} FAS V-cycle iteration:& $w_h :=\textsf{FAS}(h,S_h,w_h,b_h)$\\
  --& ENDIF &\\
 (6)& RETURN $w_h$. &\\ \hline
\end{tabular}
\end{center}

In Algorithm 2, no initial guess is required at the finest level $h$, but
we do need an initial guess at the coarsest level $h_0$ if we are going to 
solve the coarsest nonlinear system $ S_{h_0} (w _{h_0})=b_{h_0}$ using our Gauss-Seidel iteration. 
Fortunately, the coarsest nonlinear system has a very small dimension with effectively one unknown on a $3\times3\times3$ grid, the remaining values being determined by the Dirichlet boundary conditions.
Thus it is easy to obtain a good approximate solution no matter what initial guess is used.
In our implementation, we simply take the initial guess at the  coarsest level $h_0$ to be identically zero.
Often, a cubic interpolation $\hat I_H^h$ will be used in the above FMG algorithm so that
one full FMG iteration can deliver an approximation with the desired discretization accuracy.  This is useful
if the underlying FAS V-cycle multigrid iteration converges
satisfactorily, as is the case for Poisson equation solvers~\cite{Trottenberg2001}. 
Due to the strong nonlinearity of our problem,
we do not expect one FMG sweep to be sufficient. 
However, it does provide a very good initial guess for the FAS V-cycle multigrid iterations.
In both algorithms, the coarse level problem is solved approximately by performing just one nonlinear Gauss-Seidel iteration.

\section{Numerical examples}
\label{secNum}
In this section, we test our FAS multigrid algorithm using several examples available in the literature.
All simulations are implemented using MATLAB on a laptop PC
with Intel Core  i3-3120M CPU@2.50GHz and 12GB RAM.
{The CPU time (in seconds) is estimated by MATLAB's built-in timing functions \textit{tic/toc}.}
We apply one FMG iteration as an initialization step. 
Due to the nonlinearity,
this initialization step is necessary for ensuring mesh-independent convergence of the FAS algorithm.
For each FAS V-cycle, we perform only two pre-smoothing and two post-smoothing Gauss-Seidel iterations,
The coarsest level ($h_0=1/2$) problem is approximately solved using only one Gauss-Seidel iteration.
Numerical simulations indicate that the FMG-FAS algorithm has almost the same efficiency using more accurate solvers at the coarsest level. 

At the $l$-th iteration, let $u^{l}_h$ denote the computed approximation of solution and $r^l_h$ be the corresponding residual vector.
We use the pre-specified reduction in the relative residuals as the stopping criterion 
\[
  \frac{\|r_h^l\|_{2}}{\|r_h^0\|_{2}} \le 10^{-6},
\] 
where $\|\cdot\|_{2}$ denotes the discrete $L_2$ norm.
Our numerical simulations show that the chosen tolerance $10^{-6}$ is sufficient to achieve the level of discretization error.
Let $u_h$ denote the computed solution approximated at the last iteration.
We first compute the infinity norm of the approximation error
$$\textrm{Error}(h)=\|u_h-u\|_{\infty}$$
and then estimate the experimental order of accuracy by computing the logarithmic ratios of the approximation errors 
between two successive refined meshes, i.e.,
\[
 \mbox{Order}=\log_2\left(\frac{{\textrm{Error}(2h)}}{{\textrm{Error}(h)}}\right),
\]
which should be close to two provided that the scheme delivers a second-order accuracy.

In most cases, only one FMG iteration is  needed for convergence.
In special cases, we may need a few extra multigrid V-cycle iterations to fulfill the above convergence criterion.
In 3D problems, using two pre- and post-smoothing steps,
the computational cost of one FAS V-cycle is about $32/7$ the cost of a single Gauss-Seidel iteration on the finest mesh. 
One FMG-FAS iteration costs roughly $5$ times the cost of a single Gauss-Seidel iteration on the finest mesh
if we ignore the additional cost of restriction and interpolation.

When we use the nonlinear Gauss-Seidel iterations as a standalone solver, it is challenging 
to pick a good initial guess such that the iteration converges quickly. 
For the purpose of simple comparison, we always choose the slightly perturbed solution $1.01 u $
as the initial guess for the nonlinear Gauss-Seidel solver, where $u$ is  the analytic solution. 
Even using this impractical initial guess, the Gauss-Seidel solver cannot compete with the efficiency of the FMG-FAS method. {The 3D numerical comparison between the standalone Gauss-Seidel solver and the FMG-FAS method is given
in subsection \ref{Ex3D}. }
\subsection{Two dimensional examples}
\label{Ex2D}
 We begin by applying the corresponding 2D version of our FMG-FAS method
 using the nonlinear Gauss-Seidel iteration developed in~\cite{Benamou2010} as the smoother.
 {The scheme involves a 9-point stencil, and we test the Gauss-Seidel smoother using the red-black ordering and the plain lexicographic ordering, respectively.}
The FMG-FAS multigrid method delivers a far better computational efficiency then the results of~\cite{Benamou2010}, which used the Gauss-Seidel iteration as a standalone solver.
{In particular, the computation time scales almost linearly with respect to the degrees of freedom,
which is verified by a consistent fourfold increase in the CPU time {(in seconds)} as the step-size $h$ is halved.}

 \textbf{Example 1.}
  Let $\Omega=(0,1)^2$. Choose 
 \[
  f(x,y)=(1+x^2+y^2)\exp(x^2+y^2)
 \]
such that an analytic solution is
$$u(x,y)=\exp(\frac{x^2+y^2}{2}).$$

 \textbf{Example 2.}
 Let $\Omega=(0,1)^2$. Choose 
 \[
  f(x,y)=\frac{1}{\sqrt{x^2+y^2}}
 \]
such that an exact solution is
$$u(x,y)=\frac{2\sqrt{2}}{3}(x^2+y^2)^{3/4}.$$

 \textbf{Example 3.}
Let $\Omega=(0,1)^2$. Choose
\[
 f(x,y)=\frac{2}{(2-x^2-y^2)^2}
\]
such that an exact solution is
$$u(x,y)=-\sqrt{2-x^2-y^2}.$$

\begin{table}[!h]
\centering \small
\caption{Results for Example 1, 2, and 3 with the FMG-FAS algorithm (red-black ordering).}
\begin{tabular}{|c|cccc|cccc|cccc|}
\hline
 & \multicolumn{4}{c|}{Example 1}&\multicolumn{4}{c|}{Example 2}&\multicolumn{4}{c|}{Example 3} \\ \hline
$N$	&	Error & Order& Iter & CPU ($s$) &Error & Order&Iter &CPU ($s$) &Error & Order&Iter &CPU ($s$) 
\\
\hline 
128&		 4.5e-06 &--&	1& 0.1&            3.8e-05 &--&	1& 0.1&        8.9e-03 &--&	2& 0.1\\
256&		 1.1e-06 &2.0&	1& 0.1&            1.3e-05 &1.5&	1& 0.1&        6.4e-03 &0.5&	2& 0.2\\
512&		 2.8e-07 &2.0&	1& 0.3&            4.7e-06 &1.5&	1& 0.3&        4.5e-03 &0.5&	1& 0.3\\
1024&		 7.0e-08 &2.0&	1& 1.0&            1.7e-06 &1.5&	1& 1.1&        3.2e-03 &0.5&	1& 1.1\\
2048&		 1.7e-08 &2.0&	1& 3.6&            5.9e-07 &1.5&	1& 4.0&        2.3e-03 &0.5&	1& 3.8\\
4096&		 4.3e-09 &2.0&	1& 16.2&           2.1e-07 &1.5&	1& 16.5&       1.6e-03 &0.5&	1& 18.0\\
8192&		 1.0e-09 &2.1&	1& 67.1&           7.4e-08 &1.5&	1& 69.1&       1.1e-03 &0.5&	1& 70.3\\
 \hline
\end{tabular}
\label{Ex_T1_2dA}
 \end{table}

\begin{table}[!h]
\centering \small
\caption{Results for Example 1, 2, and 3 with the FMG-FAS algorithm (lexicographic ordering).}
\begin{tabular}{|c|cccc|cccc|cccc|}
\hline
 & \multicolumn{4}{c|}{Example 1}&\multicolumn{4}{c|}{Example 2}&\multicolumn{4}{c|}{Example 3} \\ \hline
$N$	&	Error & Order& Iter & CPU ($s$)  &Error & Order&Iter &CPU ($s$) &Error & Order&Iter &CPU ($s$) 
\\
\hline 
128&		 5.1e-06 &--&	1& 0.1&         4.3e-05 &--&	1& 0.1&                     8.9e-03 &--&	3& 0.1\\
256&		 1.3e-06 &2.0&	1& 0.2&         1.5e-05 &1.5&	1& 0.3&                     6.3e-03 &0.5&	2& 0.2\\
512&		 3.3e-07 &2.0&	1& 0.8&         5.4e-06 &1.5&	1& 0.9&                     4.4e-03 &0.5&	1& 0.8\\
1024&		 8.2e-08 &2.0&	1& 3.2&         1.9e-06 &1.5&	1& 3.27&                     3.2e-03 &0.5&	1& 3.1\\
2048&		 2.1e-08 &2.0&	1& 12.7&        6.8e-07 &1.5&	1& 12.9&                    2.2e-03 &0.5&	1& 12.6\\
4096&		 5.1e-09 &2.0&	1& 53.2&        2.4e-07 &1.5&	1& 54.6&                    1.6e-03 &0.5&	1& 53.5\\
8192&		 1.3e-09 &2.0&	1& 245.6&       8.5e-08 &1.5&	1& 259.4&                   1.1e-03 &0.5&	1& 250.7\\
 \hline
\end{tabular}
\label{Ex_T1_2dB}
 \end{table}
 
 For simplicity, we combine the results of Examples 1, 2, and 3  in Tables \ref{Ex_T1_2dA} and \ref{Ex_T1_2dB}.
 The CPU times manifest the desired linear time complexity of our FMG-FAS algorithm.
 The non-smoothness of the solution in Examples~2 and~3 degrade the second-order accuracy of the approximations to  $O(h^{3/2})$ and $O(h^{1/2})$,  respectively.
 Our algorithm seems to be able to effectively and efficiently handle this type of mildly singular solution. 
 {It also shows that the red-black ordering has slightly better approximation accuracy as
 well as superior computational efficiency due to the faster vectorization in  MATLAB implementation.}

 {Although we notice in our numerical experiments that only one iteration is sufficient to obtain the same approximation errors in infinity norm, for singular solutions (not even in $H^2(\Omega)$), additional FAS V-cycle iterations may be needed (e.g. in Example 3) to fulfill the given stopping criterion
 based on residual norm reduction in discrete $L_2$ norm.
  Nevertheless, in all non-degenerate ($f>0$) examples we considered, we recovered linear computational complexity as the mesh was refined.  In the fully degenerate setting ($f=0$), the ellipticity of the PDE breaks down, and the multigrid approach may become less effective in accelerating the convergence of the Gauss-Seidel iteration.}
 In that case, the deterioration of our algorithm is associated with the possible failure of the
 corresponding finite difference approximations, which are not provably convergent, and can break down in the singular case.

\subsection{Three dimensional examples}
\label{Ex3D}
Next we turn our attention to numerically solving the three-dimensional \MA equation using the Gauss-Seidel iteration
and the FMG-FAS method developed in this paper.

\textbf{Example 4~\cite{Froese2011}.}
Let $\Omega=(0,1)^3$. Choose 
\[
 f =(1+x^2+y^2+z^2)\exp\left(\frac{3(x^2+y^2+z^2)}{2}\right)
\]
such that the exact solution is
$$u =\exp(\frac{x^2+y^2+z^2}{2}).$$
In Table \ref{Ex1_T1_3d} we report the relative residual norms, infinity norm of errors, iteration numbers,  
and CPU times of the FMG-FAS algorithm and the Gauss-Seidel solver, respectively.
The `Error' column and `Order' column show that the central finite difference discretization 
indeed achieves the expected second-order accuracy in the maximum norm.
The `Iter' column implies that our FMG-FAS algorithm possesses a mesh-independent convergence.
More specifically, here `Iter'=1 implies that only one FMG iteration
is {sufficient to fulfill our stopping criteria}; no futher V-cycles are necessary.
{The achieved relative residual norms (column 'RelRes') turns out to be much lower,
thanks to the fast convergence of the FMG algorithm.}
We can do a simple calculation to see that one FMG iteration indeed costs about 5 Gauss-Seidel iterations. 
Consider the mesh size $N=32$, one Gauss-Seidel iteration takes about $324.6/867\approx 0.4$ second
and hence one FMG iteration should require approximately $5\times 0.4=2.0$ seconds, which
is almost exactly what we achieved in our algorithm.
The `CPU' column indicates that our FMG-FAS algorithm has a roughly linear time complexity
since the CPU time increases by eight times as the mesh size $h=1/N$ is halved.

On the other hand, the nonlinear Gauss-Seidel iterations as a standalone solver requires
an increasing number of iterations as the mesh size decreases, which is expected.
Even starting with an artificial initial guess that is impractically close the desired true solution, 
the nonlinear Gauss-Seidel iteration still does not produce sufficiently accurate approximations using the same stopping criterion
 (compare the row with $N=32$).
It is also interesting to notice that the relative residuals (column `RelRes') of the Gauss-Seidel solver are much larger
than that of the FMG-FAS algorithm.

For comparison, Table \ref{Ex1_T2_3d} also lists the results reported in~\cite{ObermanFroeseFast} and~\cite{Brenner2012},
where a wide-stencil hybrid finite difference {(FD)} solver and two finite element {(FE)} approximations were used
for the same test problem.
{We mention that the results in~\cite{Brenner2012} made use of a Newton solver, which required a good initial guess that was obtained by the vanishing moment method~\cite{Feng2009}.}
Though we can not fairly compare the CPU times directly
as they were computed on different machines, it is clear that the FMG-FAS method is the only one {that} scales linearly with the number of degrees of freedom.
The accuracy of the FMG-FAS approach is also competitive.
As a whole, our FMG-FAS method appears superior than both the wide-stencil finite difference solvers 
and the finite element approximations  for this type of simple problem with  a sufficiently smooth solution.

We also mention that when $N=128$, the second order accuracy is no longer observed.  This is due to the limitations of machine precision and round-off errors during the computation.  We note that the discretized \MA operator involves division by $h^6 \approx 2\times10^{-13}$ when $N = 128$.  We could get around this using higher precision arithmetic.  However, the global error is now $6\times10^{-6}$,  which is small.  The key is that we are able to resolve the data, which necessitates solving on large grids, and the resulting accuracy is sufficient for the applications we have in mind.

\begin{table}[!h]
\centering
\caption{Results for Example 4 with the FMG-FAS algorithm and the Gauss-Seidel solver.}
\begin{tabular}{|cc|ccccc|ccccc|}
\hline
 && \multicolumn{5}{c|}{FMG-FAS}&\multicolumn{5}{c|}{Gauss-Seidel} \\ \hline
$N$	&{DOFs}& RelRes &	Error & Order& Iter & CPU ($s$)  & RelRes &	Error & Order&Iter &CPU ($s$)  \\
\hline
8& 216&	3.1e-07 &	 1.2e-03&	 --& 1  & 0.1 &	9.3e-07 &	 1.2e-03&	 --& 72  & 0.4\\
16&2744&	9.5e-09 &	 2.7e-04&	 2.1& 1  & 0.2 &	9.9e-07 &	 3.0e-04&	 2.0& 270  & 11.5\\
32&27000&	2.1e-10 &	 5.7e-05&	 2.3& 1  & 2.0 &	1.0e-06 &	 8.1e-05&	 1.9& 867  & 324.6\\
64&238328&	4.1e-12 &	 1.4e-05&	 2.0& 1  & 16.9&&&&&\\
128&2000376&	5.2e-13 &	 6.2e-06&	 1.2& 1  & 136.9&&&&&\\
 \hline
\end{tabular}
\label{Ex1_T1_3d}
 \end{table}
 
\begin{table}[!h]
\centering
\caption{{Reported results for Example 4 from~\cite{ObermanFroeseFast} and~\cite{Brenner2012}}.}
\begin{tabular}{|cccc|cccc|ccc|}
\hline
 \multicolumn{4}{|c|}{Wide-stencil FD~\cite{ObermanFroeseFast}}&\multicolumn{4}{c|}{FE (quadratic) ~\cite{Brenner2012}}&\multicolumn{3}{c|}{{FE (cubic) ~\cite{Brenner2012}}} \\ \hline
$N$	& DOFs &Error &  CPU  ($s$) & $h$	& DOFs &Error &  CPU ($s$)  &DOFs &Error &CPU ($s$) \\
\hline
 8        &     343         &          1.51e-02       &     0.04&1/4     &      1581         &       9.12e-03       &         3.89     &4952 &4.14e-04 &28.78       \\
 12      &      1331       &           1.40e-02      &     0.10&1/8     &       12611       &        2.29e-03      &         38.55  &40985& 5.71e-05& 140.52         \\
 16      &      3375       &           1.29e-02      &     0.71&1/12    &       42798       &        4.49e-04      &         140.89& 140861& 7.52e-06 &874.57         \\
22      &      9261       &           1.21e-02      &     6.72&1/16    &       99436       &        2.73e-04      &         355.78  & 329244& 6.36e-06& 2758.98        \\
32      &      29791      &           1.11e-02      &     86.63&1/20    &       195110      &        1.21e-04      &         803.47  &&&       \\
 \hline
\end{tabular}
\label{Ex1_T2_3d}
 \end{table}

\textbf{Example 5.}
Let $\Omega=(0,\pi)^3$. Choose
\[
 f =(\sin(x) + 1)(\sin(y) + 1)(\sin(z) + 1)
\]
such that an exact solution is
$$u =-\sin(x)-\sin(y)-\sin(z)+(x^2+y^2+z^2)/2.$$
In Table~\ref{Ex2_T1_3d} we report the relative residual norms, infinity norm of errors, iteration numbers,  
and CPU times of the FMG-FAS algorithm and the Gauss-Seidel solver, respectively.
Again, we observed the similar excellent performance of our FMG-FAS algorithm against with the nonlinear Gauss-Seidel 
solver.
\begin{table}[!h]
\centering
\caption{Results for Example 5 with the FMG-FAS algorithm and the Gauss-Seidel solver.}
\begin{tabular}{|c|ccccc|ccccc|}
\hline
 & \multicolumn{5}{c|}{FMG-FAS}&\multicolumn{5}{c|}{Gauss-Seidel} \\ \hline
$N$	& RelRes &	Error & Order& Iter & CPU ($s$)  & RelRes &	Error & Order&Iter &CPU ($s$) \\
\hline
8&	6.8e-08 &	 1.7e-02&	 --& 2  & 0.1 &	9.2e-07 &	 1.7e-02&	 --& 82  & 0.4\\
16&	2.1e-08 &	 3.3e-03&	 2.4& 1  & 0.2 &	9.8e-07 &	 4.3e-03&	 2.0& 296  & 12.9\\
32&	2.7e-10 &	 8.1e-04&	 2.0& 1  & 2.0 &	1.0e-06 &	 1.1e-03&	 2.0& 954  & 356.2\\
64&	3.7e-12 &	 2.1e-04&	 2.0& 1  & 16.4&&&&&\\
128&	4.8e-13 &	 5.2e-05&	 2.0& 1  & 134.1&&&&&\\
 \hline
\end{tabular}
\label{Ex2_T1_3d}
 \end{table}

\textbf{Example 6~\cite{Brenner2012}.}
Let $\Omega=(0,1)^3$. Choose
\[
 f =(1/16)(x^2 + y^2 + z^2)^{-3/4}
\]
such that an analytic solution is
$$u =(1/3)(x^2+y^2+z^2)^{3/4}.
$$
In Table \ref{Ex3_T1_3d} we report the relative residual norms, infinity norm of errors, iteration numbers,  
and CPU times of the FMG-FAS algorithm and the Gauss-Seidel solver, respectively.
In constrast to the previous two examples, the solution to this problem 
is not smooth as there is a singularity at the origin.
Nevertheless, our central finite difference scheme still approximates the solutions
with a reasonable $O(h^{3/2})$ accuracy in the infinity norm.
Similarly, our FMG-FAS algorithm greatly outperforms the nonlinear Gauss-Seidel solver.
\begin{table}[!h]
\centering
\caption{Results for Example 6 with the FMG-FAS algorithm and the Gauss-Seidel solver.}
\begin{tabular}{|c|ccccc|ccccc|}
\hline
 & \multicolumn{5}{c|}{FMG-FAS}&\multicolumn{5}{c|}{Gauss-Seidel} \\ \hline
$N$	& RelRes &	Error & Order& Iter & CPU ($s$)  & RelRes &	Error & Order&Iter &CPU ($s$) \\
\hline
 8&	2.7e-07 &	 5.3e-04&	 --& 1  & 0.1 &	9.7e-07 &	 4.0e-04&	 --& 86  & 0.4\\
    16&	1.1e-08 &	 2.1e-04&	 1.4& 1  & 0.2 &	9.9e-07 &	 1.4e-04&	 1.5& 312  & 13.2\\
    32&	4.8e-10 &	 7.4e-05&	 1.5& 1  & 2.0 &	9.9e-07 &	 5.1e-05&	 1.5& 925  & 345.8\\
    64&	2.0e-11 &	 2.6e-05&	 1.5& 1  & 16.4&&&&&\\
    128&	1.2e-12 &	 9.3e-06&	 1.5& 1  & 133.9&&&&&\\
 \hline
\end{tabular}
\label{Ex3_T1_3d}
 \end{table}
 
 \textbf{Example 7~\cite{Froese2011}.}
Let $\Omega=(0,1)^3$. Choose 
\[
 f(x,y)={3}{(3-x^2-y^2-z^2)^{-5/2}}
\]
such that an analytic solution is
$$u =-\sqrt{3-x^2-y^2-z^2}.$$
In Table \ref{Ex4_T1_3d} we report the relative residual norms, infinity norm of errors, iteration numbers,  
and CPU times of the FMG-FAS algorithm and the Gauss-Seidel solver, respectively.
This example is more challenging due to the unbounded gradient at the boundary.
The `Error' column shows our scheme has a roughly $O(h^{1/2})$ accuracy in the infinity norm, which is consistent with the error in standard discretizations for the two-dimensional version of this example~\cite{Benamou2010}.
Obviously, our FMG-FAS algorithm requires significantly less CPU time than the Gauss-Seidel solver.
{We note that the slightly better accuracy achieved by the Gauss-Seidel solver is due to the unrealistic initial guess used.}
 \begin{table}[!h]
\centering
\caption{Results for Example 7 with the FMG-FAS algorithm and the Gauss-Seidel solver.}
\begin{tabular}{|c|ccccc|ccccc|}
\hline
 & \multicolumn{5}{c|}{FMG-FAS}&\multicolumn{5}{c|}{Gauss-Seidel} \\ \hline
$N$	& RelRes &	Error & Order& Iter & CPU  ($s$) & RelRes &	Error & Order&Iter &CPU ($s$) \\
\hline
  8&	1.7e-07 &	 1.4e-03&	 --& 1  & 0.1 &	8.9e-07 &	 1.4e-03&	 --& 75  & 0.4\\
   16&	2.1e-08 &	 1.4e-03&	 0.0& 1  & 0.2 &	9.7e-07 &	 1.4e-03&	 0.0& 266  & 12.2\\
   32&	2.7e-08 &	 4.3e-03&	 -1.7& 1  & 2.0 &	9.9e-07 &	 1.1e-03&	 0.3& 854  & 320.2\\
   64&	4.9e-09 &	 2.2e-03&	 1.0& 1  & 16.3 &&&&&\\
   128&	1.7e-09 &	 1.6e-03&	 0.4& 1  & 133.1&&&&&\\
 \hline
\end{tabular}
\label{Ex4_T1_3d}
 \end{table}

\section{Conclusions} 
\label{secFinal}
In this paper, we presented a fast method for solving the three-dimensional elliptic Monge-Amp\`ere equation, focusing on the Dirichlet problem and smooth solutions.  
We combined a nonlinear Gauss-Seidel iterative method with a standard centered difference discretization.
The  Gauss-Seidel iteration was then used  as an effective smoother, in combination with  a nonlinear full approximation scheme (FAS) multigrid method.
Two and three dimensional numerical examples were computed to demonstrate the computational efficiency of the proposed multigrid method, and comparison was made with other available approaches.
{In particular, we show that the computational cost of the method scales approximately linearly.  }
{Computations on a recent laptop allowed for grid sizes of $128^3$ in less than three minutes.  More importantly, since the methods scale well, implementation on more powerful computers will lead to a feasible approach to solving problems with the larger resolutions that occur in, for example three dimensional image registration} 

Future work will be to extend these ideas to filtered almost-monotone finite differences schemes in order to build solvers that apply to the singular solutions that can occur in applications.   We would also like to extend the method to the Optimal Transportation problem with applications to three dimensional image registration, or to parameter identification.

%

\section*{Acknowledgments}
The authors would like to thank the two anonymous referees for their valuable comments and suggestions that 
have greatly contributed to improving the original version of this manuscript.
The first author gratefully acknowledges the support and hospitality provided by the IMA during his participation 
in the IMA's New Directions Short Course on ``Topics on Control Theory'', which took place from May 27 to June 13, 2014.


\begin{thebibliography}{10}

\bibitem{BSNum}
G.~Barles and P.~E. Souganidis.
\newblock Convergence of approximation schemes for fully nonlinear second order
  equations.
\newblock {\em Asymptotic Anal.}, 4(3):271--283, 1991.

\bibitem{Benamou2010}
J.-D. Benamou, B.~D. Froese, and A.~M. Oberman.
\newblock Two numerical methods for the elliptic {M}onge-{A}mp\`ere equation.
\newblock {\em M2AN Math. Model. Numer. Anal.}, 44(4):737--758, 2010.

\bibitem{Benamou2014}
J.-D. Benamou, B.~D. Froese, and A.~M. Oberman.
\newblock Numerical solution of the optimal transportation problem using the
  {M}onge-{A}mp\`ere equation.
\newblock {\em J. Comput. Phys.}, 260:107--126, 2014.

\bibitem{Brandt1977}
A.~Brandt.
\newblock Multi-level adaptive solutions to boundary-value problems.
\newblock {\em Math. Comp.}, 31(138):333--390, 1977.

\bibitem{Brandt2011multigrid}
A.~Brandt and O.~Livne.
\newblock {\em Multigrid Techniques: 1984 Guide with Applications to Fluid
  Dynamics, Revised Edition}.
\newblock Classics in Applied Mathematics. SIAM, 2011.

\bibitem{Brenner2011}
S.~C. Brenner, T.~Gudi, M.~Neilan, and L.-y. Sung.
\newblock {$C^0$} penalty methods for the fully nonlinear {M}onge-{A}mp\`ere
  equation.
\newblock {\em Math. Comp.}, 80(276):1979--1995, 2011.

\bibitem{Brenner2012}
S.~C. Brenner and M.~Neilan.
\newblock Finite element approximations of the three dimensional
  {M}onge-{A}mp\`ere equation.
\newblock {\em ESAIM Math. Model. Numer. Anal.}, 46(5):979--1001, 2012.

\bibitem{Briggs2000}
W.~L. Briggs, V.~E. Henson, and S.~F. McCormick.
\newblock {\em A multigrid tutorial}.
\newblock SIAM, Philadelphia, PA, 2000.

\bibitem{CaffarelliNirenbergSpruckMA}
L.~Caffarelli, L.~Nirenberg, and J.~Spruck.
\newblock The {D}irichlet problem for nonlinear second-order elliptic
  equations. {I}. {M}onge-{A}mp\`ere equation.
\newblock {\em Comm. Pure Appl. Math.}, 37(3):369--402, 1984.

\bibitem{DelzannoChacon}
L.~Chac{\'o}n, G.~L. Delzanno, and J.~M. Finn.
\newblock Robust, multidimensional mesh-motion based on {M}onge-{K}antorovich
  equidistribution.
\newblock {\em J. Comput. Phys.}, 230(1):87--103, 2011.

\bibitem{Chen2010}
Y.~Chen and S.~R. Fulton.
\newblock An adaptive continuation-multigrid method for the balanced vortex
  model.
\newblock {\em J. Comput. Phys.}, 229(6):2236--2248, 2010.

\bibitem{Chumchob2011}
N.~Chumchob and K.~Chen.
\newblock A robust multigrid approach for variational image registration
  models.
\newblock {\em J. Comput. Appl. Math.}, 236(5):653--674, 2011.

\bibitem{delzanno2008optimal}
G.~L. Delzanno, L.~Chac{\'o}n, J.~M. Finn, Y.~Chung, and G.~Lapenta.
\newblock An optimal robust equidistribution method for two-dimensional grid
  adaptation based on {M}onge--{K}antorovich optimization.
\newblock {\em Journal of Computational Physics}, 227(23):9841--9864, 2008.

\bibitem{engquist2014application}
B.~Engquist and B.~D. Froese.
\newblock Application of the {W}asserstein metric to seismic signals.
\newblock {\em Communications in Mathematical Sciences}, 12(5), 2014.

\bibitem{Feng2013}
X.~Feng, R.~Glowinski, and M.~Neilan.
\newblock Recent developments in numerical methods for fully nonlinear second
  order partial differential equations.
\newblock {\em SIAM Rev.}, 55(2):205--267, 2013.

\bibitem{Feng2009}
X.~Feng and M.~Neilan.
\newblock Mixed finite element methods for the fully nonlinear
  {M}onge-{A}mp\`ere equation based on the vanishing moment method.
\newblock {\em SIAM J. Numer. Anal.}, 47(2):1226--1250, 2009.

\bibitem{feng2009vanishing}
X.~Feng and M.~Neilan.
\newblock Vanishing moment method and moment solutions for fully nonlinear
  second order partial differential equations.
\newblock {\em Journal of Scientific Computing}, 38(1):74--98, 2009.

\bibitem{Froese2011}
B.~D. Froese and A.~M. Oberman.
\newblock Convergent finite difference solvers for viscosity solutions of the
  elliptic {M}onge-{A}mp\`ere equation in dimensions two and higher.
\newblock {\em SIAM J. Numer. Anal.}, 49(4):1692--1714, 2011.

\bibitem{ObermanFroeseFast}
B.~D. Froese and A.~M. Oberman.
\newblock Fast finite difference solvers for singular solutions of the elliptic
  {M}onge-{A}mp\`ere equation.
\newblock {\em J. Comput. Phys.}, 230(3):818--834, 2011.

\bibitem{ObermanFroeseFiltered}
B.~D. Froese and A.~M. Oberman.
\newblock Convergent filtered schemes for the {M}onge-{A}mp\`ere partial
  differential equation.
\newblock {\em SIAM Journal on Numerical Analysis}, 51(1):423--444, 2013.

\bibitem{Gupta2000}
M.~M. Gupta and J.~Zhang.
\newblock High accuracy multigrid solution of the 3{D} convection-diffusion
  equation.
\newblock {\em Appl. Math. Comput.}, 113(2-3):249--274, 2000.

\bibitem{haber2007image}
E.~Haber and J.~Modersitzki.
\newblock Image registration with guaranteed displacement regularity.
\newblock {\em International Journal of Computer Vision}, 71(3):361--372, 2007.

\bibitem{haber20093d}
E.~Haber, G.~Pryor, J.~Melonakos, A.~Tannenbaum, et~al.
\newblock 3d nonrigid registration via optimal mass transport on the {GPU}.
\newblock {\em Medical image analysis}, 13(6):931--940, 2009.

\bibitem{haber2010efficient}
E.~Haber, T.~Rehman, and A.~Tannenbaum.
\newblock An efficient numerical method for the solution of the ${L}_2$ optimal
  mass transfer problem.
\newblock {\em SIAM Journal on Scientific Computing}, 32(1):197--211, 2010.

\bibitem{haker2004optimal}
S.~Haker, L.~Zhu, A.~Tannenbaum, and S.~Angenent.
\newblock Optimal mass transport for registration and warping.
\newblock {\em International Journal of Computer Vision}, 60(3):225--240, 2004.

\bibitem{LeVeque2007}
R.~LeVeque.
\newblock {\em Finite Difference Methods for Ordinary and Partial Differential
  Equations: Steady-State and Time-Dependent Problems}.
\newblock SIAM, Philadelphia, PA, USA, 2007.

\bibitem{Loeper2005}
G.~Loeper and F.~Rapetti.
\newblock Numerical solution of the {M}onge-{A}mp\`ere equation by a {N}ewton's
  algorithm.
\newblock {\em C. R. Math. Acad. Sci. Paris}, 340(4):319--324, 2005.

\bibitem{Motzkin}
T.~S. Motzkin and W.~Wasow.
\newblock On the approximation of linear elliptic differential equations by
  difference equations with positive coefficients.
\newblock {\em Journal of Mathematics and Physics}, 31(1):253--259, 1952.

\bibitem{Neilan2013}
M.~Neilan.
\newblock Quadratic finite element approximations of the {M}onge-{A}mp{\`e}re
  equation.
\newblock {\em Journal of Scientific Computing}, 54(1):200--226, 2013.

\bibitem{Oberman2008}
A.~M. Oberman.
\newblock Wide stencil finite difference schemes for the elliptic
  {M}onge-{A}mp\`ere equation and functions of the eigenvalues of the
  {H}essian.
\newblock {\em Discrete Contin. Dyn. Syst. Ser. B}, 10(1):221--238, 2008.

\bibitem{Ortega2000}
J.~M. Ortega and W.~C. Rheinboldt.
\newblock {\em Iterative solution of nonlinear equations in several variables}.
\newblock SIAM, Philadelphia, PA, 2000.

\bibitem{Prins2014}
C.~R. Prins, J.~H.~M. Ten Thije~Boonkkamp, J.~van Roosmalen, W.~L. Ijzerman,
  and T.~W. Tukker.
\newblock A {M}onge-{A}mp\`ere-solver for free-form reflector design.
\newblock {\em SIAM J. Sci. Comput.}, 36(3):B640--B660, 2014.

\bibitem{Saad2003}
Y.~Saad.
\newblock {\em Iterative methods for sparse linear systems}.
\newblock SIAM, Philadelphia, PA, 2003.

\bibitem{saumier2013efficient}
L.-P. Saumier, M.~Agueh, and B.~Khouider.
\newblock An efficient numerical algorithm for the {$L^2$} optimal transport
  problem with periodic densities.
\newblock {\em IMA J. Appl. Math.}, 80(1):135--157, 2015.

\bibitem{Ashmeen2011}
A.~K. Soin.
\newblock {Multigrid for Elliptic {M}onge-{A}mp\`ere Equation}.
\newblock Master's thesis, University of Waterloo, Waterloo, Ontario, Canada,
  2011.

\bibitem{Sulman2011}
M.~M. Sulman, J.~F. Williams, and R.~D. Russell.
\newblock An efficient approach for the numerical solution of the
  {M}onge-{A}mp\`ere equation.
\newblock {\em Appl. Numer. Math.}, 61(3):298--307, 2011.

\bibitem{Tadmor2012}
E.~Tadmor.
\newblock A review of numerical methods for nonlinear partial differential
  equations.
\newblock {\em Bull. Amer. Math. Soc. (N.S.)}, 49(4):507--554, 2012.

\bibitem{Trottenberg2001}
U.~Trottenberg, C.~W. Oosterlee, and A.~Sch{\"u}ller.
\newblock {\em Multigrid}.
\newblock Academic Press Inc., San Diego, CA, 2001.

\bibitem{trudinger2008monge}
N.~S. Trudinger and X.-J. Wang.
\newblock The {M}onge-{A}mpere equation and its geometric applications.
\newblock {\em Handbook of geometric analysis}, 1:467--524, 2008.

\bibitem{Zhang1998}
J.~Zhang.
\newblock Fast and high accuracy multigrid solution of the three-dimensional
  {P}oisson equation.
\newblock {\em J. Comput. Phys.}, 143(2):449--461, 1998.

\bibitem{zheligovsky2010monge}
V.~Zheligovsky, O.~Podvigina, and U.~Frisch.
\newblock The {M}onge-{A}mp{\`e}re equation: Various forms and numerical
  solution.
\newblock {\em Journal of Computational Physics}, 229(13):5043--5061, 2010.

\end{thebibliography}

\end{document}